\documentclass[12pt,a4paper]{amsart}
\usepackage{graphicx}
\usepackage{amscd}
\usepackage[mathscr]{eucal}

\usepackage[all]{xy}
\usepackage{txfonts}
\usepackage{bigints}
\usepackage{multirow}
\usepackage{mathtools}

\usepackage{latexsym}
\usepackage{amssymb}

\theoremstyle{plain}
\newtheorem{theorem}{Theorem}[section]
\newtheorem{lemma}[theorem]{Lemma}
\newtheorem{corollary}[theorem]{Corollary}
\newtheorem{proposition}[theorem]{Proposition}

\theoremstyle{definition}

\theoremstyle{remark}

\newcommand{\R}{\mathbb{R}}

\def\Ext{\operatorname{Ext}}

\def\PB{\operatorname{PB}}
\def\PO{\operatorname{PO}}
\newcommand{\lop}{\curvearrowright }

\begin{document}


\title{Bilinear forms and the $\Ext^2$-problem in Banach spaces}

\author{Jes\'us M. F. Castillo and Ricardo Garc\'ia}
\address{Departamento de Matem\'aticas, Universidad de Extremadura,
Avenida de Elvas, 06071-Badajoz, Spain}
\email{castillo@unex.es}
\email{rgarcia@unex.es}

\thanks{This research has been supported by Project MTM2016-76958-C2-1-P and Project IB16056 de la Junta de Extremadura.}
\maketitle

\begin{abstract} Let $X$ be a Banach space and let $\kappa(X)$ denote the kernel of a quotient map $\ell_1(\Gamma)\to X$. We show that  $\Ext^2(X,X^*)=0$ if and only if bilinear forms on $\kappa(X)$ extend to $\ell_1(\Gamma)$. From that we obtain i) If $\kappa(X)$ is a $\mathcal L_1$-space then $\Ext^2(X,X^*)=0$; ii) If $X$ is separable, $\kappa(X)$ is not an $\mathcal L_1$ space and $\Ext^2(X,X^*)=0$ then $\kappa(X)$ has an unconditional basis. This provides new insight into a question of Palamodov in the category of Banach spaces.\end{abstract}

\section{Introduction}

The purpose of this paper is to establish a connection between two different areas in the theory of Banach spaces: homology and holomorphy.
Let us make a brief introduction to explain the nature of our results. Given two Banach spaces $X$ and $Y$ let $\mathfrak L(X, Y )$ denote the vector
space of linear continuous operators acting between them. If $\mathfrak L$ denotes the previous functor then its first derived functor $\Ext$ is the one that assigns to each couple $X, Y$ the vector space $\operatorname{Ext}(X, Y)$ of exact sequences $0 \to Y \to \diamondsuit \to X \to 0$  modulo equivalence (see Section 2 for all unexplained terms); its second derived functor will be called $\operatorname{Ext}^2$ and its operative description can be found in Section 3. \medskip

It turns out that several important Banach space problems and results adopt the form $\operatorname{Ext}(X, Y) = 0$. For instance,
\begin{itemize}
\item Sobczyk's theorem: $\operatorname{Ext}(X, c_0) = 0$ for every separable Banach space $X$.
\item Lindenstrauss's lifting principle:  $\operatorname{Ext}({L}_1(\mu), X^*) = 0$ for  every dual space $X^*$.
\item The Johnson-Zippin's theorem: $\operatorname{Ext}(H^*, \mathcal{L}_\infty) = 0$ for every subspace $H$ of $c_0$ and every $\mathcal L_\infty$-space.
\end{itemize}
In general, a basic Banach space question is whether $\operatorname{Ext}(X, Y) = 0$ for a given couple of Banach
spaces $X, Y $; and one of the fundamental results is that $\operatorname{Ext}(\ell_2, \ell_2) \neq 0$ (see \cite{elp,kaltpeck}). Similar questions for the second derived functor $\operatorname{Ext}^2$ have not been treated too often in the literature (see \cite{wenge}). Palamodov's Problem \cite[Problem 6]{P} states: Is $\operatorname{Ext}^2(\cdot,E)=0$ for any Fr\'echet space? A solution to Palamodov's problem in the category of Fr\'echet space was provided by Wengenroth in \cite{jwenge}. In the domain of Banach spaces the answer to the question is obviously not, as can be seen in Proposition \ref{ext2}. More interesting are questions of the type: Is $\Ext^2(X,Y)=0$ for a specific choice of $X,Y$ ?. In particular, the problem of whether $\Ext^2(\ell_2, \ell_2)=0$ is open. Partial results have been obtained in \cite{cck} and \cite{ccg}. We present here the following two results. The
first one establishes an unexpected connection between homology and the study of bilinear forms:

\begin{theorem} Let $X$ be a Banach space and let $Q: \ell_1(\Gamma)\to X$ be a quotient map.
$\operatorname{Ext}^2(X, X^*)=0$ if and only if every bilinear form defined on $\ker Q$ can be extended to a bilinear form on $\ell_1(\Gamma)$.
\end{theorem}

The second result connects the $\Ext^2$ problem with the nature of subspaces of $\ell_1$. Precisely,

\begin{theorem} Let $X$ be a separable Banach space and let $q: \ell_1\to X$ be a quotient map.
\begin{enumerate}
\item If $\ker q$ is an $\mathcal{L}_1$-space then $\Ext^2(X, X^*)= 0$. 
\item If $\ker q $ has an unconditional basis and is not an $\mathcal{L}_1$-space then $\Ext^2(X, X^*) \neq 0$.
\end{enumerate}
\end{theorem}

We refer the reader to \cite{deflo} and \cite{gro} for basic and thorough information about tensor products, to \cite{HS,lane} for general homological tools and to \cite{CGJ1, CGJ2} for general results on the extension bilinear forms.

\section{$\Ext$ on Banach spaces}

Recall that a short exact sequence in the category of Banach spaces is a diagram $0 \to Y \to \diamondsuit
\to X \to 0$ formed by Banach spaces and linear continuous operators such that the kernel of each arrow coincides
with the image of the preceding. The open mapping theorem guarantees that $Y$ is a subspace of $\diamondsuit$ such that the corresponding
quotient $\diamondsuit/Y$ is $X$.
The space $\diamondsuit$ itself is called a
twisted sum of $Y$ and $X$ (in that order). Two extensions $0 \to Y
\to \diamondsuit_i \to X \to 0$ ($i=1,2$) are said to be equivalent
if there exists an arrow $T$ making
commutative the diagram
$$\begin{CD}
0@>>> Y @>>> \diamondsuit_1 @>>> X @>>> 0 \\
&&@| @VTVV @|\\
0@>>> Y @>>> \diamondsuit_2 @>>> X @>>> 0\end{CD} $$
By the 3-lemma \cite{HS}, and the open mapping theorem, $T$ must be an
isomorphism. A short exact sequence is said to split if it is equivalent to
the trivial sequence $0\to Y \to Y \oplus X \to X \to 0$. Given two Banach spaces
 $Y$ and $X$ we denote by
$\Ext(X,Y)$ the set of all possible short exact sequences $0\to Y \to \diamondsuit \to X \to 0$ modulo equivalence.\medskip

Given operators $\alpha:Y\to A$ and $\beta:Y\to B$ between Banach spaces, the associated push-out diagram is
\begin{equation}\label{po-dia1}
\begin{CD}
Y@>\alpha>> A\\
@V \beta VV @VV \overline \beta V\\
B @>> \overline \alpha > \PO
\end{CD}
\end{equation}
The push-out space $\PO=\PO(\alpha,\beta)$ is the quotient of the direct sum
$A\oplus_1 B$ by the closure of the subspace $\Delta=\{(\alpha y,-\beta y): y\in Y\}$.
The map $\overline \alpha$ is the composition of the inclusion of $B$ into $A\oplus_1 B$ and the natural
quotient map $A\oplus_1 B\to (A\oplus_1 B)/\overline\Delta$, so that
$\overline \alpha(b)=(0,b)+\overline\Delta$ and, analogously, $\overline \beta(a)=(a,0)+\overline\Delta$. All this make (\ref{po-dia1}) a commutative diagram: $\overline \beta\alpha=\overline \alpha\beta$. Suppose moreover that we are given an exact sequence $0 \longrightarrow Y \stackrel{\jmath}\longrightarrow \diamondsuit \stackrel{\rho}\longrightarrow X \longrightarrow 0$
and an operator $\tau: Y\to B$. Consider the push-out $\PO$ of the couple
$(j, \tau)$. The universal
property of the push-out gives a unique operator $\overline \rho: \PO\to X$
making a commutative diagram:
\begin{equation}\label{po-seq}
\begin{CD}
0  @>>>  Y @>j>> \diamondsuit @>\rho >> X @>>>0 \\
  &  &   @V{\tau}VV  @VV{\overline \tau}V
   @| \\
0  @>>> B @>{\overline \jmath}>> \PO  @>{\overline \rho}>> X @>>>0
\end{CD}
\end{equation} As it is well known, the lower sequence in a push-out diagram
$$
\begin{CD}
0  @>>>  Y @>{\jmath}>> \diamondsuit @>>> X @>>>0 \\
  &  &   @V{\tau}VV  @VVV
   @| \\
0  @>>> B @>>> C  @>>> X@>>> 0\end{CD}$$
splits if and only if there is an operator $T: \diamondsuit\to B$ such that $T\jmath = \tau$.

The pull-back construction \index{pull-back construction} is the dual of that of push-out in the sense of categories,
that is, ``reversing arrows''.
Indeed, let $\alpha:A\to Z$ and $\beta:B\to Z$ be operators acting between Banach spaces.
The associated pull-back diagram is
\begin{equation}\label{pb-dia}
\begin{CD}
B @> \beta >> Z\\
@A{\underline \alpha}AA @AA \alpha A\\
\PB@>>\underline \beta > A\\
\end{CD}
\end{equation}
The pull-back space is $\PB=\PB(\alpha,\beta)=\{(b,a)\in B\oplus_\infty A: \beta (b)=\alpha(a) \}$.
The underlined arrows are the restriction of the projections onto the corresponding factor. Consider an exact sequence  $0\longrightarrow Y \stackrel{\jmath}\longrightarrow \diamondsuit
\stackrel{\rho}\longrightarrow X \longrightarrow 0$ and an
operator $\tau: A\to X$. The pull-back construction yields a commutative diagram:
\begin{equation}\label{pb-seq}
\begin{CD}
0  @>>>  Y @>\jmath>> \diamondsuit @>\rho>> X @>>>0 \\
  &  & @|  @A\underline \tau AA  @AA \tau A
    \\
0@>>> Y@>\underline \jmath >> \PB  @>\underline \rho >> A@>>>0\\
\end{CD}
\end{equation}Again, as it is well known, the lower sequence in a pull-back diagram
$$
\begin{CD}
0  @>>>  Y @>>> \diamondsuit @>\rho>> X @>>>0 \\
  &  &   @| @AAA  @AA{\tau}A\\
0  @>>> Y @>>> \PB  @>>> C@>>> 0
\end{CD}
$$ splits if and only if there is an operator $T: C\to \diamondsuit$ so that $\rho T=\tau$.

\section{Projective presentations of Banach spaces}

Given a Banach space $X$ there is some index set $\Gamma$ for which there is a quotient map
 $Q: \ell_1(\Gamma)\to X$. An exact sequence
$$\begin{CD} 0 @>>> \ker Q  @>>> \ell_1(\Gamma) @>>> X @>>> 0\end{CD}$$
is usually called a projective presentation of $X$. There are many non-equivalent projective presentations of a space $X$. For
instance, if $X$ is a separable Banach space, two exact sequences $0\to \ker Q \to  \ell_1 \to X\to 0$ and
$0\to  \ell_1(\Gamma) \oplus \ker Q \to \ell_1(\Gamma)\oplus \ell_1 \to X\to 0$ define, for uncountable $\Gamma$, non-equivalent projective presentations of $X$. However, all projective presentations
are ``essentially the same" in the following sense:

\begin{proposition}\label{prop:pproje} Let $X$ be a Banach space and let $\pi: \ell_1(I) \to X$ and $Q: \ell_1(J) \to X$ be two quotient maps.
Then there are isomorphisms $\alpha, \beta$ making a commutative diagram
$$\begin{CD}
0@>>> \ell_1(J) \oplus \ker \pi  @>>> \ell_1(J)\oplus \ell_1(I) @>>> X @>>>0\\
&& @V\alpha VV @V\beta VV @|\\
0@>>> \ell_1(I) \oplus \ker Q  @>>> \ell_1(I)\oplus \ell_1(J) @>>> X @>>>0
\end{CD}$$
\end{proposition}
\begin{proof} Let $\{ (x,y): \pi x=Qy \}$ be the kernel of the quotient operator
$\rho: \ell_1(I) \oplus \ell_1(J) \to X$ given by $\rho(x,y)=
\pi x-Qy$. Since the projection onto the second coordinate $\pi_2: \ker
\rho \to \ell_1(J)$ is surjective, it admits a linear continuous
selection $s: \ell_1(J) \to \ker \rho$ given by $y \to (sy,y)$. We can define an isomorphism $\alpha: \ker \rho \longrightarrow
\ell_1(J) \oplus \ker \pi$ as $\alpha(x,y) = (x - sy, y).$ It is
well defined since $\pi(x-sy) = \pi x- \pi sy = \pi x-Qy=0$. It is obviously
injective since $\rho(x,y)=0$ implies $x=sy$ and $y=0$. And it is
surjective since $(k, y)$ is the image of $(k+sy,y)$. Hence $\ker \rho = \ell_1(J) \oplus \ker \pi$
and, analogously, $\ker \rho = \ell_p(I) \oplus \ker Q$.\end{proof}

In particular

\begin{corollary}\label{cor:corproje} Let $X$ be a separable Banach space different from $\ell_1$ and let $\pi, Q$ be
two quotient maps $\ell_1 \to X$. Then there are isomorphisms $\alpha, \beta$ making a commutative diagram
$$\begin{CD}
0@>>> \ker \pi  @>>> \ell_1 @>>> X @>>>0\\
&& @V\alpha VV @V\beta VV @|\\
0@>>> \ker Q  @>>> \ell_1 @>>> X @>>>0
\end{CD}$$
\end{corollary}
\begin{proof} Since any infinite dimensional subspace of $\ell_1$ contains a complemented copy of
$\ell_1$, one has $\ker \pi \simeq \ell_1 \oplus A \simeq \ell_1 \oplus \ell_1 \oplus A \simeq \ell_1 \oplus
\ker \pi$ and, analogously, $\ker Q \simeq \ell_1 \oplus \ker Q$. It
follows from the proof of  Proposition \ref{prop:pproje} that
$\ker Q \simeq \ell_1 \oplus \ker Q \simeq \ell_1 \oplus \ker \pi \simeq \ker \pi.$\end{proof}

Thus, regarding the results in this paper there is no difference between considering two different projective presentations of $X$ and, with a slight abuse of notation, we will simply set $\ell_1$ (instead of $\ell_1(\Gamma)$) and $\kappa(X)$ to denote ``the" kernel of a
projective presentation. Only the results in Section 6 require separability.

\section{$\Ext^2$ on Banach spaces}

Let us operatively define a few elements of the theory of the higher order
derived functors of the functor $\mathfrak L$ in Banach spaces. Given an (equivalence class of an) exact sequence $0 \to A \to B \to C\to 0$, it will be useful to give it a short name; say $F$. We will write $F:C\lop A$ when it is necessary to specify the spaces $A$ and $C$. We will also write, when necessary, $0 \to A \to B \to C\to 0 \equiv F$. The second derived space $\Ext^2(X, Y)$ is the quotient of the vector space of concatenations $FG$ in which $G: X\lop B$ and $F: B\lop Y$ with respect to the following equivalence relation. $FG \equiv F'G'$ if and only if there is a finite sequence of elements $(F_jG_j)_{j=1,\dots, n}$ so that
$$FG \longrightarrow F_1G_1 \longleftarrow F_2G_2 \longrightarrow \cdots \longleftarrow F_nG_n \longrightarrow F'G' $$
where $FG \longrightarrow F'G'$ means the existence of a commutative diagram
$$\begin{CD}
0@>>> Y@>>> A @>>> B @>>> C @>>> X @>>> 0\\
&&\Vert&& @VVV @VVV @VVV \Vert\\
0@>>> Y@>>> A' @>>> B' @>>> C' @>>> X @>>> 0
\end{CD}$$

and $FG \longleftarrow  F'G'$ means the existence of a commutative diagram

$$\begin{CD}
0@>>> Y@>>> A @>>> B @>>> C @>>> X @>>> 0\\
&&\Vert&& @AAA @AAA @AAA \Vert\\
0@>>> Y@>>> A' @>>> B' @>>> C' @>>> X @>>> 0
\end{CD}$$

Given $0 \to Y \to A \to B \to 0 \equiv F$  and $0 \to B \to C \to X \to 0\equiv G$  the element $FG\in \Ext^2(X,Y)$ is said to be $0$ if there is a commutative diagram
$$\begin{CD}
&&0 &&0\\
&&@VVV @VVV \\
&&Y@= Y\\
&&@VVV @VVV \\
0@>>> A @>>> \diamondsuit @>>> X @>>> 0\\
&& @VVV @VVV @|\\
0@>>> B@>>> C@>>> X@>>> 0\\
&&@VVV @VVV \\
&&0&& 0
\end{CD}$$

We will write $\Ext^2(X, Y)=0$ to mean that all elements of $\Ext^2(X, Y)$ are $0$.

\section{Bilinear maps on Banach spaces}
Let $E$ be a Banach space. We denote by $\mathcal B(E, \R)$ the Banach space of all scalar bilinear continuous forms on $E$. Classical theory yields the identification

$$\mathcal B(E, \R) = \mathfrak L(E, E^*).$$

Let us denote $b \to \tau_b$ (or $b_T\leftarrow T$) this identification. Precisely, $<y,\tau_b(x)> = b(x,y)$. We rescue from \cite{CGJ1, CGJ2} the following result:

\begin{lemma} A bilinear form $b$ defined on a subspace $E$ of $\ell_1$  extends to a bilinear form $B\in \mathcal B(\ell_1, \R)$ if and only if $\tau_b$ admits an operator $T: \ell_1 \to \ell_\infty$ yielding a commutative diagram
$$\begin{CD}
E @>i>> \ell_1\\
@VV{\tau_b}V @VV{T}V \\
 E^* @<<i^*< \ell_\infty
\end{CD}$$
\end{lemma}

Of course: the bilinear form that extends $b$ is $b_T$.\medskip

We introduce now the natural equivalence relation on $\mathcal B(\kappa(X), \R)$: $B \sim 0$ if and only if $B$ extends to a bilinear form on $\ell_1$. In general, $B \sim B' \Leftrightarrow B - B'\sim 0$.

\section{$\Ext^2(X, X^*)$ as a space of bilinear forms}

\begin{proposition}\label{bil} The vector spaces $\Ext^2(X, X^*)$ and $\mathcal B(\kappa(X), \R)/\sim$ are isomorphic.
\end{proposition}
\begin{proof} Let us call $0 \to \kappa(X) \to \ell_1 \to X \to 0 \equiv \Lambda_1$ (resp. $0 \to X^* \to \ell_\infty \to \kappa(X)^*\to 0 \equiv \Lambda_\infty$. Every exact sequence $0 \to B \to \diamondsuit \to X \to 0 \equiv \Omega$ is a push-out $\Omega= \phi_\Omega \Lambda_1$ for some operator $\phi_\Omega: \kappa(X) \to B$; and every exact sequence $0 \to X^* \to \diamondsuit \to B \to 0 \equiv \Omega$ is a pull-back
$\Omega=\Lambda_\infty\psi_\Omega$ for some operator $\phi_\Omega: B \to \ell_\infty/X^*$.\medskip

The isomorphism between $\Ext^2(X, X^*)$ and $\mathcal B(\kappa(X), \R)/\sim$ is as follows: given  $FG\in \Ext^2(X, X^*)$, with $G: X\lop B$ and $F: B\lop X^*$, write

$$FG = F \phi_G \Lambda_1 = \Lambda_\infty \psi_{F\phi_G} \Lambda_1$$

where $\psi_{F\phi_G}: \kappa(X) \to \kappa(X)^*$ is the operator associated to a bilinear form on $\kappa(X)$. Conversely, given a bilinear form $b$ on $\kappa(X)$ with associated operator $\tau_b: \kappa(X) \to \kappa(X)^*$ we form the element $\Lambda_\infty \tau_b \Lambda_1$.\medskip

This correspondence is compatible with the equivalence relations: the commutative diagram:
$$\begin{CD}
0@>>> X^*@>>> A @>>> B @>>> C @>>> X @>>> 0 &\equiv FG\\
&&\Vert&& @AAA @AA{\psi_G}A @AAA \Vert\\
0@>>> X^*@>>> \PB @>>> \kappa(X) @>i>> \ell_1 @>q>> X  @>>> 0 &\equiv F\psi_G \Lambda_1 \\
&& \Vert&& @VVV @VV{\tau_b}V @VVV \Vert\\
0@>>> X^* @>>> \ell_\infty @>i^*>> \ell_\infty/X^* @>>> \PO @>>>
X  @>>> 0& \equiv \Lambda_\infty \tau_b \Lambda_1
\end{CD}$$
shows that $FG$ and $\Lambda_\infty \tau_b \Lambda_1
$ are the same element of $\Ext^2(X,
X^*)$. And, $\Lambda_\infty \tau_b \Lambda_1\equiv 0$ if and only if
the exact sequence $0 \to \ell_\infty/X^* \to \PO \to X\to
0 \equiv \tau_b \Lambda_1$ splits, which occurs if and only if
$\tau_b$ admits an extension to an operator $\tau: \ell_1 \to
\ell_\infty/X^*$.
Since $\ell_1$ is projective, this operator can be lifted to an operator
$T: \ell_1 \to \ell_\infty$ through the quotient map $i^*$ yielding a commutative diagram

$$\xymatrix{
\kappa(X) \ar[r]^{i} \ar[d]_{\tau_b} &\ell_1 \ar[dl]^\tau\ar[d]^{T}\\
 \ell_\infty/X^* & \ell_1^* \ar[l]^{i^*}.\\
}$$

Therefore, the bilinear form $b$ on $\kappa(X)$ extends to the bilinear form $b_T$ on $\ell_1$.\medskip

Conversely, if $b$ extends to a bilinear form $B$ on $\ell_1$ then $T= i^* \tau_B$ is an extension of $\tau_b$ and thus
$\tau_b \Lambda_1 \equiv 0$ which, in particular, implies $FG \equiv \Lambda_\infty \tau_b \Lambda_1 \equiv 0$. \end{proof}

This proves Theorem 1. A direct consequence is that we obtain a different homology sequence to define $\Ext^2$: given a projective presentation $0\to \kappa(X) \to \ell_1 \to X \to 0$ then one has an exact sequence

$$\begin{CD}
0@>>> \mathcal B(X, \R) @>>> \mathcal B(\ell_1, \R) @>>> \mathcal B(\kappa(X), \R) @>>>\Ext^2(X,X^*) @>>>0
\end{CD}$$

\section{Projective tensors}

Let  $X\widehat{\otimes}_{\pi} Y$ denote the tensor product  endowed with the projective tensor norm, so that (see \cite[3.2]{deflo})
$(X\widehat{\otimes}_{\pi} Y)^* = \mathfrak L(X, Y^*) = \mathfrak L (Y,X^*)$. It is plain that bilinear
forms defined on $\kappa(X)$ can be extended to bilinear forms on $\ell_1$  if and only if  the restriction
operator $R:  \mathfrak L(\kappa(X),\kappa(X)^*)\rightarrow \mathfrak L (\ell_1,\ell_1^*)$ is surjective, which happens if and only if
$\imath\otimes \imath$ is an into isomorphism. One thus has:

\begin{proposition}\label{Equiv} $\;$ Let $X$ be a separable Banach space and let $\imath: \kappa(X)\to \ell_1$ be the canonical inclusion. The following are equivalent
\begin{enumerate}
\item $\Ext^2(X, X^*)=0$.
\item All bilinear forms defined on $\kappa(X)$ can be extended to bilinear forms on $\ell_1$.
\item The restriction
operator $R:  \mathfrak L(\kappa(X), \kappa(X)^*)\rightarrow \mathfrak L(\ell_1, \ell_1^*)$ is surjective.
\item $ \imath \otimes \imath: \kappa(X) \widehat{\otimes}_{\pi} \kappa(X)  \longrightarrow   \ell_1\widehat{\otimes}_{\pi} \ell_1$
is an into isomorphism.
\end{enumerate}
\end{proposition}

Included in the proof are the quantitative facts: if all bilinear forms defined on $\kappa(X)$ can be extended to bilinear forms on $\ell_1$
then there is a constant $C$ so that all bilinear norm one forms can be extended to bilinear forms with norm at most $C$. Which means that $ \imath \otimes \imath: \kappa(X) \widehat{\otimes}_{\pi} \kappa(X)  \longrightarrow   \ell_1\widehat{\otimes}_{\pi} \ell_1$
is an into $C$-isomorphism, and conversely.\medskip

Recall that $\mathcal{L}_1$-spaces preserve the projective tensor norm (see \cite[3.]{deflo}), therefore $\Ext^2(X, X^*)=0$ whenever $\kappa(X)$
is an $\mathcal{L}_1$-space. This proves Theorem 2 (1).\medskip

It is well known that $\ell_1$ contains uncountably many non-isomorphic $\mathcal L_1$ spaces \cite{LP} and that $X$ does not have to be an $\mathcal L_1$ space when $\kappa(X)$ is an $\mathcal L_1$-space. There are therefore many nontrivial examples of spaces $X$ so that $\Ext(X, X^*)=0$.

\section{Unconditional bases and the extension of multilinear forms}
Throughout this section all Banach spaces will be separable. A beautiful classical result of Lusky \cite{lus} shows that whenever $X$ has a basis, $\kappa(X)$ has a basis. See \cite{castmo} for further generalizations of this result. In general, $\kappa(X)$ need not to have an unconditional basis when $X$ has an unconditional basis, as the case of $X=c_0$ shows (as it follows from \cite[Cor. 2.2]{john}); while it is not known whether $\kappa(\ell_2)$ has an unconditional basis. And this is relevant to our discussion because of the following two results:
\begin{itemize}
\item Lindenstrauss and Pelzy\'nki proved in \cite{LP} that if $X$ is an $\mathcal{L}_1$-space with unconditional basis then $X$ is isomorphic to $\ell_1$.
\item Defant et al. show in \cite{DGMP} that if $Y$ is a space with unconditional
basis that is a subspace of an $\mathcal{L}_1$-space and there is a constant $C$ such that every $n$-linear form $\tau$ on $Y$ extends to an $n$-linear form $T$ on the whole space satisfying an estimate $\|T\|\leq  C^n \|\tau\|$ then $Y = \ell_1$.
\end{itemize}

With all this we are ready to obtain our second result.

\begin{theorem}\label{dual} Let  $\kappa(X)$ be subspace of $\ell_1$ that is not an $\mathcal{L}_1$-space. If $\kappa(X)$ has  an unconditional basis, then  $\Ext^2(X, X)^*\neq0$.
\end{theorem}
\begin{proof} As we know, if $\Ext^2(X, X^*)=0$ then $\imath \otimes \imath: \kappa(X) \widehat{\otimes}_{\pi}  \kappa(X) \hookrightarrow \ell_1 \widehat{\otimes}_{\pi} \ell_1$ is an into isomorphism. Let $C$ be its norm. Then $ \otimes^n \imath:  \widehat{\otimes}^n_{\pi} \kappa(X)  \longrightarrow   \widehat{\otimes}^n_{\pi} \ell_1$ is  an  into $C^n$-isomorphism for all $n$ as it follows from the particular properties of $\ell_1$ which make it sufficient to make extensions ``one variable at each time":
$$
\kappa(X) \widehat{\otimes}_{\pi} \kappa(X)  \hookrightarrow \ell_1 \widehat{\otimes}_{\pi} \kappa(X)=\ell_1 \left(\kappa(X) \right) \hookrightarrow \ell_1 \left(\ell_1\right) =\ell_1\widehat{\otimes}_{\pi} \ell_1$$

and then iterate the argument

$$\widehat{\otimes}^n_{\pi} \kappa(X)  \hookrightarrow \ell_1 \widehat{\otimes}_{\pi}\kappa(X)\widehat{\otimes}_{\pi} \cdots \widehat{\otimes}_{\pi}\kappa(X) =\ell_1 \left(\widehat{\otimes}^{n-1}_{\pi} \kappa(X)\right) \hookrightarrow  \ell_1\left(\widehat{\otimes}^{n-1}_{\pi}\ell_1\right) =  \widehat{\otimes}^n_{\pi}  \ell_1.$$

Thus, $n$-linear norm one forms on $\kappa(X)$ extend to $n$-linear forms on $\ell_1$ with norm at most $C^n$. If $\kappa(X)$ has unconditional basis then the result of Defant et al. in \cite{DGMP} yields that $\kappa(X)=\ell_1$, which is impossible.\end{proof}

It is in this way that the problem of whether $\Ext^2(\ell_2, \ell_2)=0$ connects with the classical unsolved problem of whether $\kappa(\ell_2)$ has an unconditional basis:

\begin{corollary} If $\Ext^2(\ell_2, \ell_2)=0$ then $\kappa(X)$ does not have an unconditional basis.
\end{corollary}

As we mentioned in the Introduction, explicit solutions to Palamodov's question in the category of Banach spaces can be easily obtained.
\begin{proposition}\label{ext2} $\Ext^2(\cdot, \ell_2)\neq 0\neq \Ext^2(\ell_2, \cdot)$.\end{proposition}
\begin{proof} Let $0\to \ell_2\to X \to \ell_2 \to 0$ be any nontrivial twisted sum of Hilbert spaces (see, e.g., \cite{elp, kaltpeck}). Embed $\ell_2$ into $L_1 = L_1(0,1)$ and form the element
$$\begin{CD}
0@>>> \ell_2 @>>> X @>>> \ell_2 @>>> L_1(0,1)  @>>> L_1(0,1)/\ell_2@>>>0
\end{CD}$$
It cannot be $0$ because Lindenstrauss lifting principle yields $\Ext(L_1(0,1), X)=0$ since $X$ is reflexive; thus,
if the element is $0$ then $\ell_2$ will be complemented in $X$, which is impossible. Also, if one writes $\ell_2$ as a quotient of $\ell_\infty$ and forms the element
$$\begin{CD}
0@>>> K @>>> \ell_\infty @>>> \ell_2 @>>> X @>>> \ell_2@>>>0
\end{CD}$$
this cannot be $0$ simply because $\ell_\infty$ is injective and, thus, if the element is $0$ then $\ell_2$ would be complemented in $X$.\end{proof}

We can also obtain an explicit example of $X$ so that $\Ext^2(X, X^*)\neq 0$ 

\begin{proposition}\label{ext22} If $X= \ell_1/\ell_1(\ell_2^n)$ then $\Ext^2(X, X^*)\neq 0.$\end{proposition}
\begin{proof} Pick a subspace $\ell_1(\ell^n_2)$ of $\ell_1(\ell_1)=\ell_1$ (a subspace $\ell_1(\ell^n_p)$ for any $1< p \leq 2$ will also work  \cite{js}). This subspace clearly has an unconditional basis ---it can even be chosen so that $\ell_1/\ell_1(\ell_2^n)$ fails to enjoy the Bounded Approximation Property; see the final example in \cite{castmo}--- and is not an $\mathcal{L}_1$-space since its dual $\ell_\infty(\ell_2^n)$ cannot be an $\mathcal L_\infty$ space because it contains $\ell_2^n$ uniformly complemented. Use now Theorem \ref{dual}.\end{proof}

\end{document}